\documentclass[11pt]{amsart}
\usepackage{amssymb}
\usepackage{amsmath}
\usepackage{setspace}
\usepackage{amsthm}
\usepackage{epsfig,color,colordvi,wasysym}
\usepackage{graphicx}
\usepackage{amsfonts}
\usepackage{transparent}
\mathchardef\mhyphen="2D

\newtheorem{theorem}{Theorem}

\newtheorem{lemma}[theorem]{Lemma}

\newtheorem{comment}[theorem]{Comment}

\newcommand\commentout[1]{}

\newcommand\Def[1]{{\bf #1}}

\renewcommand{\H}{\mathcal{H}}

\newcommand{\ZZ}{\mathbb{Z}}
\newcommand{\RR}{\mathbb{R}}
\def\a{\mathbf{a}}

\def\0{\mathbf{0}}
\def\1{\mathbf{1}}
\def\x{\mathbf{x}}
\def\y{\mathbf{y}}
\def\d{\mathbf{d}}

\renewcommand\vec{\overrightarrow} 

\begin{document}

\title{Parking functions, Shi arrangements, and mixed graphs}  

\author[Beck]{Matthias Beck}
\address{Department of Mathematics, San Francisco State University, San Francisco, CA 94132, USA}
\email{mattbeck@sfsu.edu}

\author[Berrizbeitia]{Ana Berrizbeitia}
\address{Department of Mathematics, University of Iowa, Iowa City, IA 52242, USA}
\email{ana-berrizbeitia@uiowa.edu}

\author[Dairyko]{Michael Dairyko}
\address{Department of Mathematics, Iowa State University, Ames, IA 50011, USA}
\email{mdairyko@iastate.edu}

\author[Rodriguez]{Claudia Rodriguez}
\address{School of Mathematical and Statistical Sciences, Arizona State University, Tempe, AZ 85287, USA}
\email{Claudia.Rodriguez.3@asu.edu}

\author[Ruiz]{Amanda Ruiz}
\address{Department of Mathematics and Computer Science, University of San Diego, San Diego, CA 92110, USA}
\email{amruiz@hmc.edu}

\author[Veeneman]{Schuyler Veeneman}
\address{Department of Operations Research and Information Engineering, Cornell University, Ithaca, NY 14850, USA}
\email{sav63@cornell.edu}




\date{7 September 2014}

\thanks{We thank Ricardo Cortez and the staff at MSRI for creating an ideal research environment at MSRI-UP, and Brendon Rhoades, Tom Zaslavsky, and three anonymous referees for helpful comments and suggestions.
This research was partially supported by the NSF through the grants DMS-1162638 (Beck) and DMS-1156499 (MSRI-UP REU), and by the NSA through grant H98230-11-1-0213. 
A portion of this work was completed while A.\ Ruiz was a postdoctoral fellow at Harvey Mudd College, supported in part by the NSF (DMS-0839966).}

\maketitle

\begin{abstract}
The \emph{Shi arrangement} is the set of all hyperplanes in $\RR^n$ of the form $x_j - x_k = 0$ or $1$ for $1 \le
j < k \le n$.  Shi observed in 1986 that the number of regions (i.e., connected components of the complement) of
this arrangement is $(n+1)^{n-1}$. An unrelated combinatorial concept is that of a \emph{parking function}, i.e.,
a sequence $(x_1, x_2, ..., x_n)$ of positive integers that, when rearranged from smallest to largest, satisfies
$x_k \le k$. (There is an illustrative reason for the term \emph{parking function}.) It turns out that the number
of parking functions of length $n$ also equals $(n+1)^{n-1}$, a result due to Konheim and Weiss from 1966. A
natural problem consists of finding a bijection between the $n$-dimensional Shi arragnement and the parking
functions of length $n$. Pak and Stanley (1996) and Athanasiadis and Linusson (1999) gave such (quite different)
bijections. We will shed new light on the former bijection by taking a scenic route through certain mixed graphs.
\end{abstract}


\section{Introduction}

Our goal is to draw (bijective) connections between three seemingly unrelated concepts; their names
form the title of our paper, and we start by introducing them one by one.


\subsection{Parking Functions}

Imagine a one-way street with $n$ parking spots and a cliff at its end. We'll give the first parking spot the number
1, the next one number 2, etc., down to the last one, number $n$. Initially they're all free, but
there are $n$ cars approaching the street, and they'd all like to park.
To make life interesting, every car has a parking preference, and we record the preferences in a
sequence; e.g., if $n=3$, the sequence $(2, 1, 1)$ means that the first car would like to park at
spot number 2, the second car prefers parking spot number 1, and the last car would also like to
park at number 1. The street is narrow, so there is no way to back up. Now each car enters the
street and approaches its preferred parking spot; if it is free, it parks there, and if not, it
moves down the street to the first available spot. We call a sequence a \Def{parking function} (of
length $n$) if all cars end up happily finding a parking spot, i.e., none fall off the cliff.
For example, the sequence $(2, 1, 1)$ is a parking function (of length 3), whereas the sequence
$(1,3,3,4)$ is not (see Figure~\ref{fig:nonparkingfunction}).

\begin{figure}[h]
\includegraphics[scale=.3]{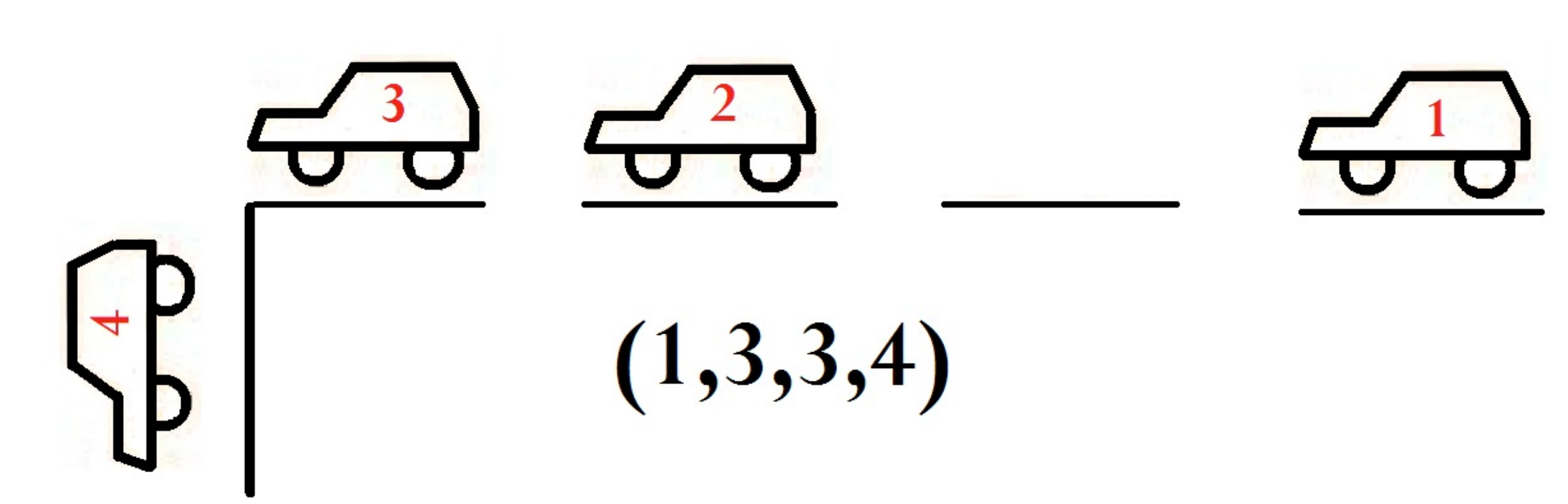}
\caption{A sequence that is not a parking function.}\label{fig:nonparkingfunction}
\end{figure}

The following enumerative result makes for a fun (nontrivial) exercise in any undergraduate
combinatorics class. The earliest reference we are aware of is \cite[Section 5]{pyke}, though its
explicit relation to parking functions first appeared in~\cite{konheimweiss}.
 
\begin{theorem}\label{thm:parkingfunctions}
There are precisely $(n+1)^{ n-1 }$ parking functions of length~$n$.
\end{theorem}

One possible proof of this theorem starts with the following equivalence, which we will need below:

\begin{lemma}\label{lem:parkingfunctionperm}
A sequence $\x \in \ZZ_{ >0 }^n$ is a parking function if and only if $\x$ is componentwise less than
or equal to some permutation of $(1, 2, \dots, n)$.
\end{lemma}

This equivalent notion also explains why parking functions naturally appear in many contexts; these include tree inversions \cite{kreweras}, symmetric functions \cite{haimanparking}, Riemann--Roch theory for graphs (where
parking functions are also called \emph{reduced divisors} \cite{bakernorine}), Hopf algebras \cite{novellithibon}, chip-firing games (where parking functions go by the name of \emph{superstable configurations} \cite{holroydlevinemeszarosetal}), and vertex operators~\cite{dotsenko}.


\subsection{Shi Arrangements}

A \Def{hyperplane arrangement} is a finite collection of hyperplanes in some Euclidean space, i.e., sets of the form
\[
  \left\{ \x \in \RR^n : \, \a \, \x = b \right\} 
\]
for some $\a \in \RR^n \setminus \{ \0 \}$ and $b \in \RR$.
A famous example is the $n$-dimensional \Def{(real) braid arrangement} consisting of all hyperplanes of the form $x_j = x_k$
for $1 \le j < k \le n$; it has a natural connection to the symmetric group $S_n$, and indeed, the
braid arrangement forms a geometric bridge between various algebraic and combinatorial concepts.
A picture (in this case, Figure \ref{fig:braid}) is worth a thousand words.

\begin{figure}[htb]
\def\JPicScale{.8}
\def\bm{$}
\def\em{$}
\begin{center}
\input{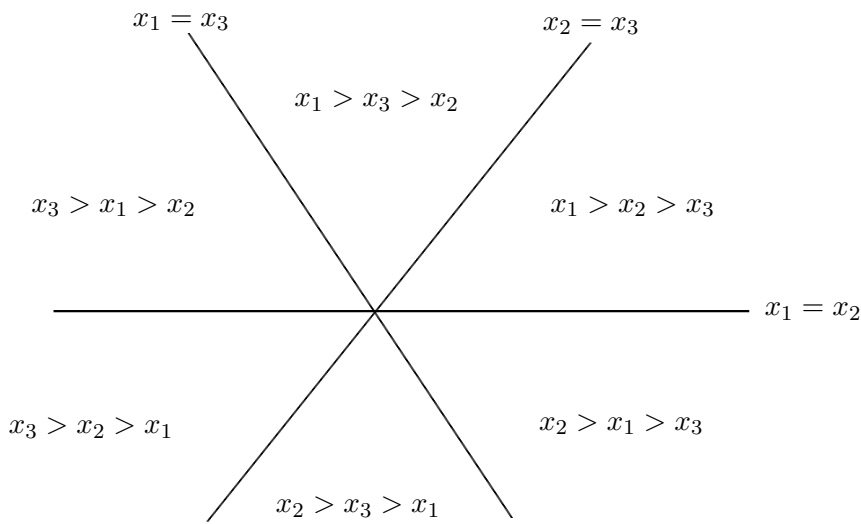}
\end{center}
\caption{The bijection between permutations and regions of the braid arrangement illustrated for $n=3$.}\label{fig:braid}
\end{figure}

The $n$-dimensional \Def{Shi arrangement} is a close relative to the braid arrangement: it consists of all hyperplanes of the form
\[
  x_j - x_k = 0
  \qquad \text{ and } \qquad
  x_j - x_k = 1
  \qquad \text{ for all } 1 \le j < k \le n \, .
\]
A \Def{region} of the hyperplane arrangement $\H$ is a maximal connected component of $\RR^n \setminus
\bigcup \H$. The following result was first proved by Shi~\cite{shi}.

\begin{theorem}\label{thm:shi}
The $n$-dimensional Shi arrangement has precisely $(n+1)^{ n-1 }$ regions.
\end{theorem}

One proof of this theorem (first given in \cite{headley}, see also \cite{athanasiadisfinitefieldmethod}) is through
the \emph{characteristic polynomial} of the Shi arrangement, which carries more information than the
number of its regions \cite{zaslavskythesis}; e.g., with it one can also prove that the $n$-dimensional
Shi arrangement has precisely $(n-1)^{ n-1 }$ \emph{bounded} regions.

Hyperplane arrangements are ubiquitous in several areas of mathematics and form a research area in its own right.
Some highlights include connections between combinatorics (starting with \cite{zaslavskythesis}), matroid theory \cite{orientedmatroids}, topology (starting with \cite{hattori}), commutative algebra \cite{orliksolomon}, and algebraic geometry \cite{saito}.
We refer to \cite{orlikterao,stanleyhyparr} for further study.


\subsection{Parking Graphs}

A \Def{mixed graph} $G$ is an amphibian: $G$ is between a graph and a directed graph, in the sense that $G$
may contain both undirected and directed edges (see, e.g., \cite{hararypalmer}). 
Mixed graphs have various applications, e.g., to scheduling problems.
Here we are only interested in a fairly special class of mixed graphs.
As with directed graphs, we define the \Def{in-degree} of a vertex $v$ of $G$ to be the number of directed edges pointing into $v$; the \Def{out-degree} of $v$ is the number of
directed edges pointing away from $v$. 

A \Def{parking graph} $P$ is a mixed graph with vertex set $[n] := \left\{ 1, 2, \dots, n \right\} $
whose underlying graph is the complete graph $K_n$ and whose edges
satisfy what we will call the \emph{source-sink condition}. To describe it, we use the following
nomenclature for $1 \le j < k \le n$:
\begin{itemize}
\item a directed edge $j \leftarrow k$ in $P$ is a \Def{ down edge};
\item a directed edge $j \rightarrow k$ in $P$ is an \Def{ up edge};
\item an undirected edge $jk$ in $P$ is a \Def{downish edge}.
\end{itemize}
(See Figure \ref{fig:example} for examples.)
One reason for the last terminology is that we will associate to each parking graph $P$ a directed
graph $\vec P$ which, in addition to the directed edges of $P$, also contains each undirected edge
$jk$ of $P$ as a directed edge $j \leftarrow k$.
For example, the mixed graph in Figure \ref{fig:example} gives rise to a coherently oriented 3-cycle. 

\begin{figure}[htb]
\begin{center}
\includegraphics[scale=.4]{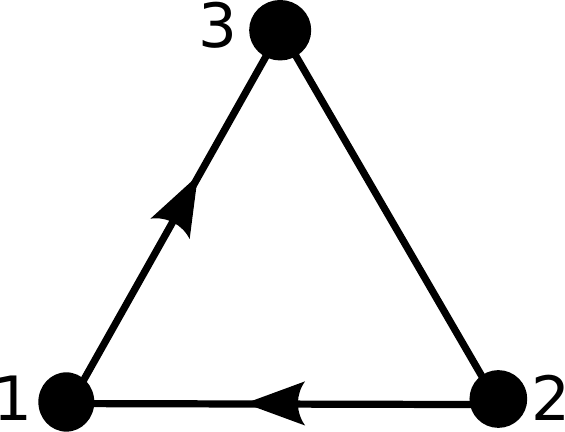}
\caption{A down ($1 \leftarrow 2$), up ($1 \rightarrow 3$), and downish ($23$) edge.}\label{fig:example}
\end{center}
\end{figure}

The \Def{source-sink condition} says that $\vec P$ is acyclic, i.e., it contains no coherently oriented cycles,
and for any triangle $G$ (complete subgraph on 3 vertices) of $P$ that has both a down and a downish edge, the source (i.e., the
vertex with in-degree 0) and sink (the vertex with out-degree 0) of $\vec G$ cannot be connected by a downish edge. 

\begin{theorem}\label{thm:parkinggraph}
There are precisely $(n+1)^{ n-1 }$ parking graphs on the vertex set~$[n]$. 
\end{theorem}

Both the notion of a parking graph and Theorem \ref{thm:parkinggraph} seem to be new, though close relatives of parking graphs can be found in \cite{hopkinsperkinson}.
We will give their motivation next.


\subsection{A natural question}

Theorems \ref{thm:parkingfunctions} and \ref{thm:shi} beg the problem of finding a bijection between the parking functions of length $n$ and the regions of the $n$-dimensional Shi arrangement.
The first such bijection, due to Pak and Stanley \cite{stanleyshiproceedings}, recursively labels the regions of the Shi arrangement with parking functions from the inside out, 
starting by giving the label $(1, 1, \dots, 1)$ to the ``central'' region $\left\{ \x \in \RR^n : \, x_n + 1
> x_1 > x_2 > \dots > x_n \right\}$,
and giving a rule how the labels change as one crosses any of the Shi hyperplanes.
A different bijection, due to Athanasiadis and Linusson \cite{ath}, constructs a diagram of nonnesting arcs on $[n]$ from a region of the Shi arrangement in $\RR^n$,
and each diagram in turn determines a parking function.

Our goal is to give a third bijection by way of parking graphs, hence proving Theorem \ref{thm:parkinggraph}: we will
exhibit a bijection between Shi regions and parking graphs and a bijection between parking graphs and parking
functions. This implies, of course, a bijection between Shi regions and parking functions, and its forward
direction turns out to be equivalent to the Pak--Stanley bijection. (Pak and Stanley ``only'' gave
an injection from the set of Shi regions in $\RR^n$ to the set of parking functions of length $n$
and then appealed to the fact that both sets are equinumerous.)

Hopkins and Perkinson \cite{hopkinsperkinson} recently extended the Pak--Stanley bijection to \emph{bigraphic arrangements}, using mixed graphs as a similar
intermediate tool. Our construction is not identical but equivalent---our parking graphs correspond to Hopkins--Perkinson's \emph{Shi-admissable
orientations} of a complete graph, and our source-sink condition is equivalent to their condition of ``having no bad cycles." At any rate, our approach seems more direct (which should not be surprising, since Hopkins and Perkinson proved a more general result).
Even more recent work of Backman also essentially contains parking graphs; his definition adds a $q$-rooted spanning tree~\cite[Lemma 5.6]{backman}.


\subsection{Cayley's formula}

As a historical aside, the arguably most famous counting instance with answer $(n+1)^{ n-1 }$ is the number of labeled trees with $n+1$ vertices.
Thus a similar natural question concerns possible bijections between labeled trees and, say, parking functions. The oldest such bijection
we are aware of is due to Sch\"utzenberger \cite{schutzenbergerparking}; the arguably simplest bijection is due to Foata and Riordan \cite{foatariordan} and illustrates that bijections between labeled trees and parking
functions are generally easier than those between parking functions and regions of the Shi arrangements: The Foata--Riordan bijection starts with encoding a labeled tree by its \emph{Pr\"ufer code}, a recursively constructed sequence
containing the label of the vertex incident to the lowest-labeled leaf (which then gets removed to yield the recursion; this procedure stops when a single edge is left).\footnote{Pr\"ufer's idea \cite{pruefer} immediately
gives a bijection between labeled trees with $n+1$ vertices and sequences of length $n-1$ containing numbers in $[n+1]$, thus confirming Cayley's formula.}
Foata and Riordan use the following map, which they attribute to Henry O.\ Pollak, from the set of all parking functions of length $n$ to the set of all Pr\"ufer codes encoding labeled trees with $n+1$ vertices:
\begin{equation}\label{eq:pollak}
  \left( x_1, x_2, \dots, x_n \right) \mapsto \left( x_2 - x_1 , x_3 - x_2 , \dots, x_n - x_{n-1} \right) \bmod n+1 \, .
\end{equation}
This map can be inverted if $x_1$ is known independently, and the Foata--Riordan proof shows that for any Pr\"ufer code, there exists a unique $x_1$ that gives an
inverse of \eqref{eq:pollak} yielding a parking function.



\section{The Bijections}

A directed complete graph $G$ (also called a \emph{tournament}) contains a cycle if and only if there is a triangle subgraph in $G$ that is a cycle. As a consequence,
 it suffices to limit acyclicity in the source-sink condition to triangles. 

Now we tackle the actual bijections. We start with a map from parking graphs to regions of the Shi
arrangement: given a parking graph $P$ on the vertex set $[n]$, we define
\begin{equation}\label{eq:defpsi}
  \psi(P) := \left\{ \x \in \RR^n : \, 
               \begin{array}{rl}
                       x_j - x_k \ge 1 & \text{ if } \ j \leftarrow k \ \text{ is down, } \\
                 0 \le x_j - x_k \le 1 & \text{ if } \ j k \ \text{ is downish, } \\
                       x_j - x_k \le 0 & \text{ if } \ j \rightarrow k \ \text{ is up }
               \end{array}
             \right\} .
\end{equation}
Thus $\psi(P)$ is a potential region of the $n$-dimensional Shi arrangement. We will now show that it is
an honest region, i.e., full-dimensional, and that every such region comes from a parking graph.
Figure \ref{fig:3dimex} shows the case $n=3$.

\begin{figure}[htb]
\def\JPicScale{1}
\def\bm{$}
\def\em{$}
\begin{center}
\input{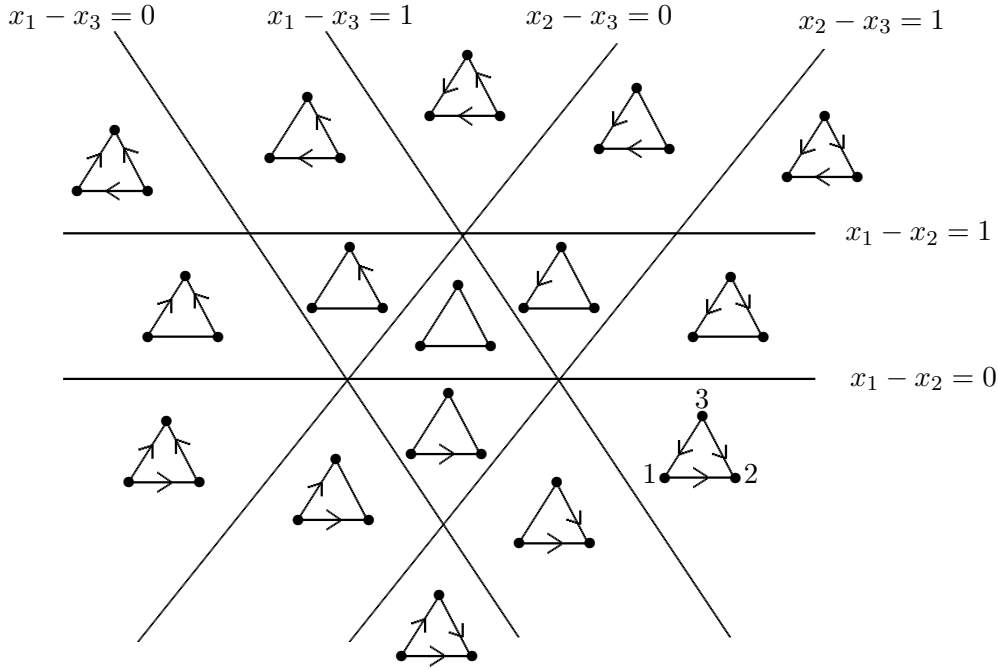}
\end{center}
\caption{The bijection between parking graphs on three vertices (all labeled as in the graph in the bottom right)  and regions of the 3-dimensional Shi arrangement.}\label{fig:3dimex}
\end{figure}

\begin{theorem}\label{thm:graphtoshi}
The map $\psi$ gives a bijection between the parking graphs on the vertex set $[n]$ and the regions of
the $n$-dimensional Shi arrangement.
\end{theorem}

\begin{proof}
For a region $R$ of the $n$-dimensional Shi arrangement, we define (by a slight abuse of notation but in
light of what's to come) $\psi^{ -1 }(R)$ to be the mixed graph one obtains by reversing the
recipe given by \eqref{eq:defpsi}. Our goal is to show that both $\psi$ and $\psi^{ -1 }$ are well-defined
maps from the set of all parking graphs on the vertex set $[n]$ to the set of all regions of the $n$-dimensional Shi
arrangement and back. Since $\psi$ and $\psi^{ -1 }$ are inverses by construction, this will prove Theorem~\ref{thm:graphtoshi}.

It is slightly easier to see that $\psi^{ -1 }$ is well defined, so let's start with that: namely, we
need to show that, for any region $R$, the mixed graph $\psi^{ -1 }(R)$ satisfies the source-sink
condition.  If  $\overrightarrow{\psi^{ -1 }(R)}$ contains a cycle, then for some $j<k<m$ the region $R$ satisfies either
\[\begin{array}{rcl} \left\{ \begin{array}{r}x_j-x_k>0\\x_k-x_m>0\\x_m-x_j>0\end{array} \right\} &\text{ or }& \left\{
\begin{array}{r}x_j-x_k<0\\x_k-x_m<0\\x_m-x_j<0 \end{array} \right\} . \end{array}\]
The first set of inequalities gives rise to the contradiction $x_j>x_j$ and the second to $x_j<x_j$. 

If  $\psi^{ -1 }(R)$ violates the downish part of the source-sink condition, then the defining inequalities for $R$ include both $0
\le x_j - x_k \le 1$ (where $k$ is the source and $j$ the sink of some triangle in $\overrightarrow{\psi^{ -1 }(R)}$) and $x_j - x_k \ge 1$ (see Figure \ref{fig:sourcesinkcond}); but a region has nontrivial interior, so no such $R$ can exist.

\begin{figure}[htb]
\begin{center}\def\svgwidth{3cm}
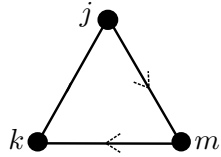
\end{center}
\caption{Violation of the source-sink condition; here $k<m<j$  and at least one of the two potential  down edges is present.}\label{fig:sourcesinkcond}
\end{figure}

To see that $\psi$ is well defined, we need to show, given a parking graph $P$, that $\psi(P)$ has
nontrivial interior. Let's look at the in-degree sequence of $\vec P$; since $\vec P$ is acyclic, the in-degree sequence is a permutation $\left(
\sigma(0), \sigma(1), \dots, \sigma(n-1) \right)$ of the vector $(0, 1, \dots, n-1)$. Note that this
in-degree sequence also means that $j \to k$ is an  up edge if and only if $\sigma(j-1) < \sigma(k-1)$.
Now choose $\x \in \RR^n$ whose coordinates  satisfy
\[
  x_{ \sigma^{ -1 } (0) + 1 } < x_{ \sigma^{ -1 } (1) + 1 } < \dots < x_{ \sigma^{ -1 } (n-1) + 1 } \, . 
\]
We claim that $\x$ is in the interior of $\psi(P)$.
Indeed, $\x$ satisfies all inequalities $x_j \ge x_k$ in \eqref{eq:defpsi} strictly, by construction. In words,
$\x$ respects both the inequality constraints stemming from an  up edge and those stemming from a
down/downish edge. What remains to be shown is that there is no potential contradiction from the
interplay of the constraints stemming from a down and a downish edge. But the only way such a
contradiction could occur is when
\[
  \begin{array}{ccccccccc}
  0 & < & x_{ \sigma^{ -1 } (i) + 1 } &                             & - &                             & x_{ \sigma^{ -1 } (m) + 1 } & < & 1 \\
    &   &                             & x_{ \sigma^{ -1 } (j) + 1 } & - & x_{ \sigma^{ -1 } (k) + 1 } &                     & > & 1
  \end{array}
\]
for some $m <i,~k < j $ and $ \sigma^{ -1 } (i) \le \sigma^{ -1 } (j)< \sigma^{ -1 } (k) \le \sigma^{ -1 } (m) $. This means that $P$ has a subgraph $G$ of
the form in Figure \ref{fig:submixedgraph} (where $m$ and $k$ or $j$ and $i$ could coincide). 
\begin{figure}[htb]
\begin{center}\def\svgwidth{3cm}
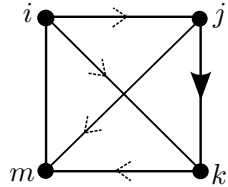
\end{center}
\caption{A sub-mixed graph violating the source-sink condition. Vertex $i$ is the source and vertex $m$ is the sink. The edges between $i$ or $m$ and $j$ or $k$ may be up, down, or downish edges in the direction of the dashed arrows.}\label{fig:submixedgraph}
\end{figure}
%
In the subgraph $G$, vertex $i$ is the source, $m$ is the sink, the edge $im$ is downish, and $k\leftarrow j$ is a  down edge.  We
will show that $G$ violates the source-sink condition. Notice that
$k$ is the sink of  the triangle on vertices $i,j,$ and $k$, and that $j$ is the source of the triangle on vertices $j,k$ and $m$.
If  the edge $ki$  is downish, then the triangle on vertices $i,j$, and $k$ violate the  source-sink condition.   If the edge $k\leftarrow i$ is a  down edge, then the triangle on vertices $k, i$, and $m$  violates the source-sink condition.  Finally, if  $i\rightarrow k$ is an  up edge, then $m<i<k<j$. Therefore the edge $m\leftarrow j$ must be  a  down edge, thus violating the source-sink condition for the triangle on vertices $i,j,$ and $m$.   
\end{proof}

Our second map is between parking graphs and parking functions: given a parking graph $P$, let
$\d(P)$ be its in-degree sequence, and define
\[
  \phi(P) := \d(P) + \1
\] where $\bf 1$ $\in\RR^n$ is a vector all of whose coordinates are 1. 

\begin{theorem}\label{thm:bijparkgraphfunction}
The map $\phi$ gives a bijection between the parking graphs on the vertex set $[n]$ and the parking functions of length $n$.
\end{theorem}

\begin{proof}
To show that $\phi$ is well defined, note once more that, given a parking graph $P$, the in-degree sequence of $\vec{P}$ is a
permutation of $(0,1,\ldots,n-1)$, and thus  $\d(P)$ is component-wise less than or equal to this permuted vector;  this implies that $\d(P)+\1$ is a parking function.

To show that $\phi^{ -1 }$ exists, we provide an algorithm that constructs a parking graph $P(\x)$ for the parking function $\x$ and show that the algorithm yields $\phi^{-1}$. Starting with the vertex set $[n]$, the algorithm determines when to add up edges from a particular vertex and when to add down edges. In contrast with the function
$\phi$, the algorithm is not as easy to define nor understand. We start with a detailed example which will shed some light into the process behind the algorithm.\\

\noindent
{\it Example.} To see the algorithm in practice, we will construct $P(3,1,1,2)$, the parking graph associated to the parking function $(3,1,1,2)$. 
The output of the algorithm must be a parking graph with vertex set $\{1,2,3,4\}$ and in-degree sequence $\d(P)=(2,0,0,1)$.
 There are many graphs with such an in-degree sequence.  Figure \ref{fig:algexnon} shows two different labeled complete mixed graphs on four vertices. The graph on the left violates the source-sink condition: triangle $\{1,2,4\}$ violates the downish part of the source-sink condition, and triangle $\{2,3,4\}$ creates a cycle. The graph on the right is the unique parking graph associated to $(3,1,1,2)$. 

\begin{figure}[htb]
\begin{center}
\def\svgwidth{3cm}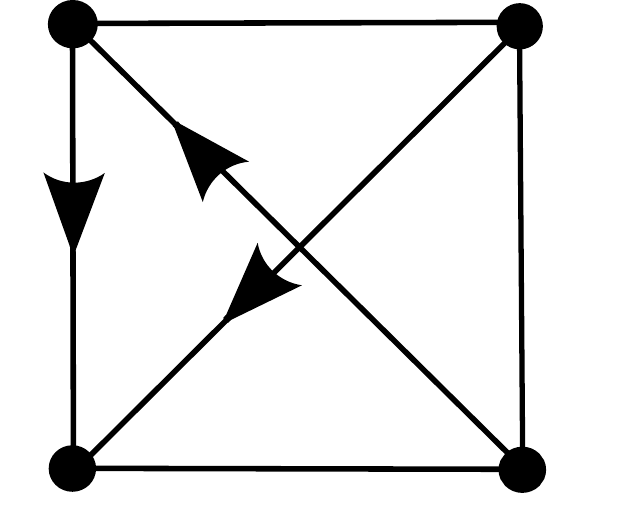
\hspace{2cm}
\def\svgwidth{3cm}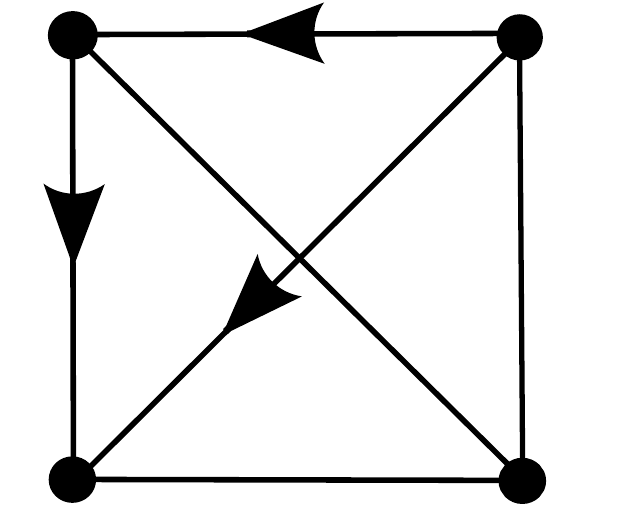
\end{center}
\caption{Two graphs with the same in-degree; the graph on the right is $P(3,1,1,2)$.}\label{fig:algexnon}
\end{figure}
 
So how can we create an algorithm that yields an actual parking graph? We begin the process by analyzing the entries of the input with smallest value. By virtue of being a
parking function, the input will always have at least one entry with a value of $1$. In our example, we have two such entries: the ones corresponding to vertex $2$ and vertex
$3$. These will be the two vertices with in-degree $0$ in our graph. Suppose the first up edge introduced was $2 \rightarrow 4$. Then the edge $34$ must be downish. There is
also a downish edge $23$ since both have in-degree $0$. Then the triangle $\{2,3,4\}$ creates a cycle, as seen on the left graph of Figure  \ref{fig:algexnon}. Therefore, we instead begin by creating an up edge $3 \rightarrow 4$ as seen in the first graph in Figure \ref{fig:algex}. In general, in order to avoid creating cycles, up edges will always emerge from the available vertex corresponding to the entry with highest index and lowest value. 
This process is formalized as the \emph{up step} of the algorithm below.

\begin{figure}[htb]
\begin{center}
\def\svgwidth{3.5cm}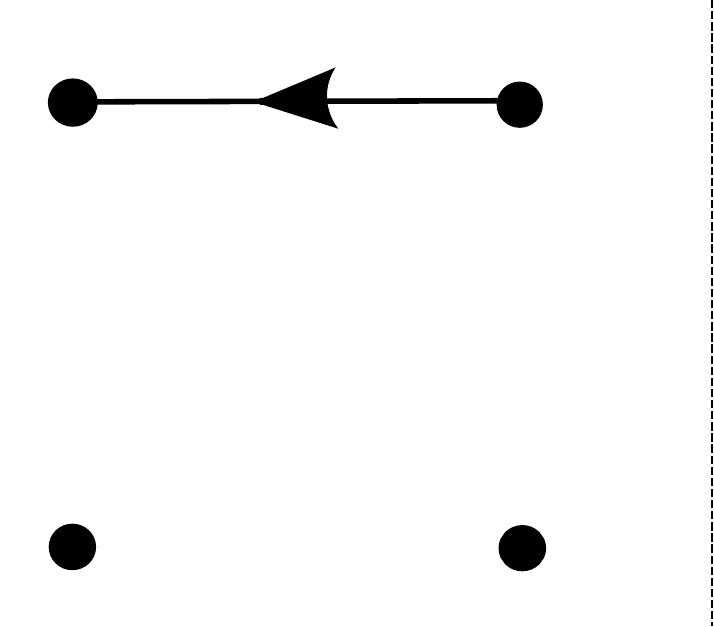
\def\svgwidth{3.5cm}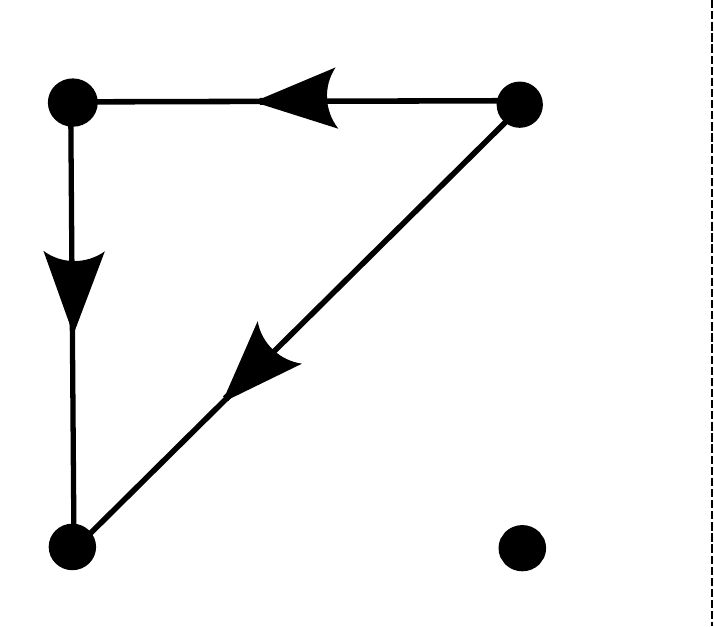
\def\svgwidth{3.35cm}\input{algexample3112.pdf_tex}
\end{center}
\caption{The algorithm in practice: building the parking graph  $P(3,1,1,2)$.}\label{fig:algex}
\end{figure}

So we now have a graph where vertex $4$ has in-degree 1, as desired. Since vertices $2, 3$ and $4$ all have the correct in-degree, we no longer introduce up edges. Vertex $1$ must have in-degree $2$, so we need to introduce two down edges from vertices 2, 3, or 4. 
%
%
%
%
If the down edge were $1 \leftarrow 2$, then either the edge between $1$ and $4$ or the edge between $1$ and $3$ will be downish.  If $14$ is a  downish edge, then the triangle
$\{1,2,4\}$ would violate the source-sink condition, as $4$ would be the source and $1$ would be the sink. This can also be seen in the left graph of Figure \ref{fig:algexnon}. An analogous argument applies if $13$ was the downish edge. It is
therefore necessary to have the  down edges $4 \leftarrow 1$ and $3 \leftarrow 1$ as seen in the middle graph in Figure \ref{fig:algex}. In general, in order to avoid violating
the downish part of the source-sink condition, down edges will always emerge from the available vertices with the corresponding entries with the lowest in-degree. This process
is formalized as the \emph{down step} of the algorithm below.

%

Each of the vertices has the desired in-degree, and so we add all other edges as downish edges, finally yielding $P(3,1,1,2)$.

We now formalize the ideas introduced in this example with an algorithm:

\begin{eqnarray*}
&{}& \text{Input: parking function } \x \in \ZZ_{ > 0 }^n \\
&{}& \y := \x - (1, 1, \dots, 1) \\
&{}& (\star) \text{ If there exists } y_k = 0 \text{ then } \\
&{}& \qquad \ \, j := \max \left\{ k : \, y_k = 0 \right\} \\
&{}& \qquad \left. \begin{array}{l}
            \text{introduce } j \to k \\
            y_k := y_k - 1
            \end{array} \right\} \text{ for all } k > j \text{ with } y_k > 0 \qquad [{\bf up \hspace{1.5mm} step}] \\
&{}& \qquad \ \, y_j := y_j - 1 \\
&{}& \qquad \ \, y_k := y_k - 1 \text{ for all } y_k < 0 \\
&{}& \qquad \ \, \text{go to } (\star) \\ 
&{}& \text{Else if there exists } y_k > 0 \text{ then } \\
&{}& \qquad \ \, \text{choose minimal } y_j \text{ such that there exists }k<j \text{ with } y_k>0 \text{ and } k\nleftarrow j\\
&{}& \qquad \left. \begin{array}{l}
           \text{introduce } k \leftarrow j \\
            y_k := y_k - 1
            \end{array} \right\} \text{ for all } k < j \text{ with } y_k > 0 \qquad [{\bf down \hspace{1.5mm} step}] \\
&{}& \qquad \ \, \text{go to } (\star) \\ 
&{}& \text{Else [all } y_k < 0 {]} \\ 
&{}& \qquad \ \, \text{introduce remaining edges as undirected and stop } \\
&{}& \text{Output: mixed graph } P(\x), \text{ \Def{source priority vector} } s(\x):=(y_1,\ldots,y_n)
\end{eqnarray*}

Some additional terminology will be useful in order to talk about the algorithm:
We will  refer to $\hat\y:=\x-(1,\ldots,1)$ as the \Def{initial} $\y$. At each iteration of the algorithm, we will refer to the current value of the $j^{th}$ entry of $y$ as $y_j$.  If $j<k$, we say the $j$-entry is
\Def{before} the $k$-entry of $\y$, and the $k$-entry is \Def{after} the $j$-entry of $\y$.
If $y_j=0$ and $j=\max\{k:y_k=0\}$ we call $j$ the \Def{feeder} of the corresponding up step, or an \Def{up feeder}. Similarly, $j$ is the
\Def{feeder} of a down step, or a \Def{down feeder}, if a  down edge $k\leftarrow j$ is introduced. 
We will say $j$ is a \Def{down feeder candidate} for $k$ whenever the algorithm is at a down step, $y_k>0$, $k<j$ and $k \nleftarrow j$.
In Lemma \ref{lem:sourceorderup}, we will show the source priority vector $s(\x)$ associates a total order  to the vertices of $P(\x)$ that determines  when each vertex is the feeder of an up step. If $|s_i(\x)|>|s_j(\x)|$, then $i$ became a feeder of an up step before ~$j$ did. 

We must check that this algorithm both runs without failing and indeed returns the parking graph associated to the input, the parking function $\x$. Note that the algorithm always starts at an up step because the input $\x$ is a parking function. Hence, there exists $k$ such that $x_k=1$, so $\hat{y}_k$ will be 0. This ensures, in particular, that the algorithm will always start. 
If there exists a $y_k$  with $y_k=0$, an up step will always occur. So the up step of the algorithm is well defined. However, the down step of the algorithm could cause problems. How do we know that when there exists a $y_k>0$ that triggers a down step, there also exists a down feeder candidate, $j$, for $k$?
 We answer this question in the following lemma.

\begin{lemma}\label{lem:feeder}
If the algorithm is at a down step, then for some $y_k>0$, there exists a down feeder candidate for $k$, namely, $j$. Furthermore, if $j$ is the down feeder for $k$, then $y_j<0$.
\end{lemma}

\begin{proof}
Suppose, by way of contradiction, that a down feeder candidate did not exist. By assumption, at least one $y_k>0$ and for every other $l \ne k$, either $y_l>0$ or $y_l<0$, but there is no entry with $y_l=0$ since the algorithm is at a down step. Let $k_0=\min\{k: y_k>0\}$. If there is no down feeder candidate, then either there is no $j$ with $k<j$ for any $y_k>0$, in which case $k_0=n$, or for all $j$ with $k_0<j$ we have that $k_0 \leftarrow j$ has already been introduced. 

In the first case, $k_0=n$, $y_{k_0}=y_n>0$ and $y_l<0$ for all $1\le l<n$. It follows that at some point in the algorithm, each of the $l$ must have fed $n$ at the step when they were 0. Therefore, $n$ has been fed $n-1$ times during an up step, and yet $y_n \ge 1$. Since each up step redefines $y_n:=y_n-1$, if we trace back the steps of the algorithm, it follows that for the initial $\hat \y$, $\hat y_n \ge 1+(n-1)$. Therefore, $x_n = n+1$, which means $\x$ was not a parking function by Lemma \ref{lem:parkingfunctionperm}, so this case is not possible.

In the second case, for all $j$ with $k_0<j$, $k_0\leftarrow j$ has already been introduced, so $k_0$ has been fed by every entry after $k_0$. Furthermore, for all $l<k_0$, $y_l<0$, so at some point in the algorithm $l$ fed $k_0$ during an up step, so $k_0$ has been fed by every entry before $k_0$. So again, $k_0$ has been fed $n-1$ times, so by the same argument as in the first case, tracing our steps back would give us that the initial input was not a parking function. So this case is also not possible. 

This proves the existence of a down feeder candidate at a down step. We now show that if $j$ is the feeder, i.e., the minimal $y_j$ amongst the candidates, then $y_j<0$. 
 
Recall that there are no 0 entries in $\y$ since we are at a down step. Again, by way of contradiction, suppose that there are no negative candidates for feeders. This means that for any $p$ with $y_p>0$, each $q>p$ with $y_q<0$ has already fed $p$; that is, the edge $p\leftarrow q$ has already been introduced in a previous  down step.  Furthermore, the algorithm gives us that for any $r<p$ with $y_r<0$, $r$ has fed $p$ in a previous up step.  Let $k$ be the total number of positive entries in the current down step. Then for any of the $k$ possible entries $y_p$ with $y_p>0$, $p$ has been fed by all $n-k$ indices with negative entries.  

Since the input $\x$ is a parking function and $y_p>0$,  we have $1\le y_p\le n-1-(n-k)=k-1$.  
Since we have gone through at least $n-k$ iterations of the algorithm, we can trace the algorithm back $n-k$ steps to see that for the initial $\hat\y$, $1+(n-k) \le\hat y_p\le
n-1$. Therefore, $n-k+2\le x_p\le n$. Since there are exactly $k$ entries $x_p$ satisfying this inequality, it follows that there are only $n-k$ entries of $\x$ that could be
less than $n-k+2$, and by Lemma  \ref{lem:parkingfunctionperm}, $\x$ cannot be a parking function. This contradiction proves Lemma~\ref{lem:feeder}. 
\end{proof}

Lemma \ref{lem:feeder} shows that the algorithm is well defined, i.e., it runs without failing. Since the input is a finite $n$-tuple with non-negative entries, at some point in
the algorithm each of these entries will eventually become 0 and later negative after a series of up steps. Therefore, the final output of the algorithm $s(\x)$ is indeed an
$n$-tuple consisting of all negative entries. Furthermore, the following lemma tells us that these entries are all distinct and that the $n$-tuple gives us meaningful
information about the desired parking function $P(\x)$.

\begin{lemma}\label{lem:sourcepri}
 The source priority vector $s(\x)$ is a total order that indicates the reverse order in which indices are the feeder of an up step.
\end{lemma}
\begin{proof}
The algorithm begins with an up step with the rightmost index $j$ of $\hat \y$ with $\hat y_j=y_j=0$ as the feeder. This up step of the algorithm instructs us to do as follows: for each $l>j$ with $y_l>0$, introduce the edge $j \rightarrow l$; redefine $y_j:=y_j-1$; and redefine $y_l:=y_l-1$. Since there are no negative entries in $\hat \y$, we need not do anything else in this step. So the new entries of $\y$ are $y_j=-1$, $y_l \ge 0$ for all $l \ne j$. Therefore, $j$ is the unique entry with $y_j=-1$ after the first step of the algorithm.

Suppose we continue with the algorithm and are now at another up step. Let $j$ be the rightmost index of $\y$ with $y_j=0$. The up step instructs us to change the entries of $\y$ as follows:
redefine $y_j:=-1$; for every $y_l>0$  with up edge $j \rightarrow l$ introduced, redefine $y_l := y_l -1$; for any $y_k<0$, redefine $y_k:=y_k-1$; and for any $i\ne j$ such
that $y_i=0$, no new adjacent edges are introduced, so $y_i$ remains $0$. Thus for $i,j,k,$ and $l$ playing the above roles, the new entries of $\y$ are $y_j=-1$, $y_l \ge 0$,
$y_k \le -2$, and $y_i=0$. Therefore, the $j$-entry is the unique entry with $y_j=-1$. 

It follows inductively that $s(\x)$ gives a total ordering of $\{-1,\ldots,-n\}$, and since $y_k<0$ gets redefined to be $y_k-1$ only during up steps, this ordering signifies the order in which $j$ was the feeder of an up step.
\end{proof}

It is now possible to understand exactly when the algorithm will create an up edge $i \rightarrow j$ in terms of the source priority vector. This is summarized in the following lemma.

\begin{lemma}\label{lem:sourceorderup}
Let $i<j$. Then $|s_i(\x)|>|s_j(\x)|$ if and only if $i\rightarrow j$ in $P(\x)$. 
\end{lemma}
\begin{proof}
Suppose $i\rightarrow j$ in $P(\x)$. Then $i\rightarrow j$ is introduced when $y_i=0$ and $y_j>0$. Therefore, $y_i$ becomes $-1$ before $y_j$ becomes $-1$. Hence $|s_i(\x)|>|s_j(\x)|$.   

Suppose $i\nrightarrow j$. Then, during the iteration when $y_i=0$, it must have been true that $y_j<0$. Therefore $|s_i(\x)|<|s_j(\x)|$. 
This proves Lemma \ref{lem:sourceorderup}.
\end{proof}

It is also useful to talk about the order in which the algorithm introduces down edges in terms of the source priority vector. 

\begin{lemma}\label{lem:sourceorderdown}
Let $j$ and $k$ both be down feeder candidates for $i$ at a certain down step. Then $j$ becomes the down feeder for $i$ at this step if and only if $|s_j(\x)|>|s_k(\x)|$.
\end{lemma}
\begin{proof}
Lemma \ref{lem:feeder} gives us that if $j$ is a down feeder for $i$ then $y_j<0$. Since the down feeder is the minimum $y_j$ amongst the candidates, it follows that $j$ is the down feeder for $i$ if and only if $y_j<0$ and $y_j<y_k$. This is true if and only if $|s_j(\x)|>|s_k(\x)|$.
\end{proof}

It remains to show that if $\x$ is the input of the algorithm, then the output, $P(\x)$, is indeed a parking graph (i.e., $P(\x)$ satisfies the source-sink condition), and $P(\x)$ has in-degree sequence $\x -1$. Therefore, the algorithm yields $\phi^{-1}$.

To show that $\d(P(\x))=\x-1$, we will follow an entry $y_k$ through the algorithm. The initial value $\hat y_k=x_k-1$ is the desired in-degree of vertex
$k$ of $P(\x)$.  If $y_k>0$, then  a directed edge is introduced towards $k$ if and only if $y_k$ is reduced by $1$.
Therefore, when $y_k=0$, vertex $k$ of $P(\x)$ has in-degree $x_k-1$; and once $y_k\le 0$, the only edges incident with $k$ that are introduced are edges that are directed away
from $k$ or that are undirected, neither of which contributes to the in-degree. Therefore, $\d(P(\x))=\x-1$, as desired.

To show $P(\x)$ is acyclic, suppose $1\le i<j<k\le n$ and $i\rightarrow j \rightarrow k$. When the edge $j\rightarrow k$ is introduced, $y_j=0$ while $y_k>0$. 
 When $i\rightarrow j$  is introduced, then $y_i=0$ and $y_j>0$. Therefore $y_k>0$, so $i\rightarrow k$ is
also introduced. Thus there is no cycle $i\rightarrow j\rightarrow k\rightarrow i$. Now suppose $i\rightarrow k$. When the edge $i\rightarrow k$ is introduced, $y_i=0$, $y_k>0$ and for all $l>i$, $y_l\ne0$. In particular, $y_j\ne0$. If $y_j>0$, then  $i\rightarrow j$ is
introduced in the same step.  Suppose $y_j<0$; then an up step previously occurred in which $j$ was the source and $y_k$ was positive, so $j\rightarrow k$ was introduced. Therefore there are no cycles involving $i\rightarrow k$.

Lastly, we show that $P(\x)$ satisfies the downish part of the source-sink condition, and this will complete the proof that $P(\x)$ is a parking graph with the correct in-degree sequence.
Consider a triangle $\Delta$ of $P(\x)$ with vertices $i,j,k$ with $i<j<k$. If $P(\x)$ has either an up edge or a  down edge between each pair of vertices, then there is no violation of the source-sink condition and we're done. 
So suppose for some pair of vertices of $\Delta$, $P(\x)$ has a downish edge between them. We consider three cases.

Case 1: There is a downish edge $ij$. Suppose, by way of contradiction, that $\Delta$ violates the downish part of the source-sink condition. Then $\Delta$ must have an  up edge $j \rightarrow k$ and a  down edge $i \leftarrow k$. Since $j<k$, Lemma \ref{lem:sourceorderup} gives us that $|s_j(\x)|>|s_k(\x)|$. On the other hand, since $k \leftarrow i$ is introduced before $j \leftarrow i$, Lemma \ref{lem:sourceorderdown} gives us that $|s_k(\x)|>|s_j(\x)|$. So this cannot happen.

Case 2: There is a downish edge $jk$. Suppose, by way of contradiction, that $\Delta$ violates the downish part of the source-sink condition. Then $\Delta$ must have an up edge $i \rightarrow j$ and a  down edge $i \leftarrow k$. Since $i<j$, Lemma \ref{lem:sourceorderup} gives us that $|s_i(\x)|>|s_j(\x)|$. In particular, when $y_i=0$, $y_j>0$, so when $y_i>0, y_j>0$. On the other hand, $k$ is a down feeder for $i$, but since $y_j$ was also positive when $y_i$ was positive, $k$ would have been a down feeder for $j$ during this same step. This contradicts the assumption that $jk$ was downish, so this cannot happen.

Case 3: There is a downish edge $ik$. Suppose, by way of contradiction, that $\Delta$ violates the downish part of the source-sink condition. Then all edges of $\Delta$ are either down or downish, and at least one of  $i\leftarrow j$ or $j\leftarrow k$ is in $P(\x)$.  Suppose  $i\leftarrow j$ is in $P(\x)$.  Then during a down step in the algorithm, both $j$ and $k$ were down feeder candidates for $i$, and $j$ became the feeder for $i$, so by Lemma \ref{lem:sourceorderdown}, $|s_j(\x)|>|s_k(\x)|$, and since $j<k$, by Lemma \ref{lem:sourceorderup}, $\Delta$ has up edge $j \rightarrow k$. This contradicts the assumption that $\Delta$ violates the source-sink condition. Now Suppose $j\leftarrow k$ is in $P(\x)$.  Since $i \nrightarrow j$, by Lemma \ref{lem:sourceorderup}, $|s_i(x)|<|s_j(x)|$. Therefore, during the down step in the algorithm where $k$ was a down feeder for $j$, it would have also fed $i$ as both $y_i$ and $y_j$ were simultaneously positive at this step. This contradicts the assumption that $ik$ is downish. Therefore, this cannot happen either.

With these three cases exhausted, we conclude that $P(\x)$ satisfies the downish part of the source-sink condition. Since $P(\x)$  satisfies the source-sink condition and
$\d(P(\x))=\x-1$, we conclude that $P(\x)$ is a parking graph and the algorithm yields $\phi^{-1}$, and this finishes the proof of Theorem~\ref{thm:bijparkgraphfunction}.
\end{proof}


\section{Closing Remarks}

As mentioned above, the composition $\phi \circ \psi^{ -1 }$ gives a bijection between the regions
of the $n$-dimensional Shi arrangement and the parking functions of length $n$, and the ``forward''
direction of this bijection gives the map introduced by Pak and Stanley
\cite{stanleyshiproceedings}. It is not clear to us how parking graphs interact with the other known
bijection by Athanasiadis and Linusson \cite{ath}.
In the same paper, Athanasiadis and Linusson study an arrangement that is situated between braid and Shi arrangement, namely, the arrangement
\begin{align*}
  x_j - x_k &= 0
  \qquad \text{ for all } 1 \le j < k \le n \, , \\
  x_j - x_k &= 1
  \qquad \text{ for all } 1 \le j < k \le n \text{ with } jk \in E \, ,
\end{align*}
where $E$ is the edge set of a given fixed graph on $n$ vertices (see also
\cite{athanasiadisfinitefieldmethod}). It would be interesting if parking graphs could shed any further light upon these arrangements.

Finally, it could be interesting to study possible connections of this work with the chromatic theory for gain graphs initiated by
Berthom\'e, Cordovil, Forge, Ventos, and Zaslavsky; see, in particular, \cite[Section~8]{berthomeetal}.



\bibliographystyle{amsplain}  
\bibliography{MyBibliography}  

\end{document}